\documentclass[english]{article}

\usepackage{microtype}
\usepackage{babel}
\usepackage[T1]{fontenc}
\usepackage[utf8]{inputenc}
\usepackage{lmodern}
\usepackage{csquotes}
\usepackage{mathtools}
	\numberwithin{equation}{section}
\usepackage{amsthm}
\usepackage{amssymb}
\usepackage{amscd}
\usepackage[hidelinks]{hyperref}
\usepackage{zref-clever}
	\zcsetup{abbrev=true}
\usepackage{booktabs}

\newcommand{\BC}{\mathbb C}
\newcommand{\BZ}{\mathbb Z}
\newcommand{\BP}{\mathbb P}
\newcommand{\cm}{\mathcal M}
\newcommand{\cn}{\mathcal N}
\newcommand{\cL}{\mathcal L}
\newcommand{\tK}{{\widetilde K}}
\newcommand{\ta}{{\widetilde a}}
\newcommand{\tn}{{\widetilde n}}
\newcommand{\tN}{{\widetilde N}}
\newcommand{\tM}{{\widetilde M}}
\newcommand{\tm}{{\widetilde\mu}}
\newcommand{\la}{\lambda}
\newcommand{\ep}{\epsilon}
\newcommand{\cp}{\mathsf P}
\newcommand{\ch}{\operatorname{ch}}
\newcommand{\td}{\operatorname{td}}
\newcommand{\End}{\operatorname{End}}
\newcommand{\Hom}{\operatorname{Hom}}
\newcommand{\Ext}{\operatorname{Ext}}
\newcommand{\Hilb}{\operatorname{Hilb}}
\newcommand{\vdim}{\operatorname{vdim}}

\newtheorem{proposition}{Proposition}[section]

\title{Counting higher-rank sheaves on toric threefolds}
\author{Nicolo Piazzalunga, Artan Sheshmani and Shing-Tung Yau}
\date{}

\begin{document}

\maketitle
\begin{abstract}
We develop an equivariant K-theoretic vertex formalism for a colored, abelianized higher-rank sector on smooth projective toric threefolds $X$.  Nonzero first Chern classes are encoded as magnetic fluxes and absorbed into line-bundle twists, after which the virtual character is redistributed into finite perturbative, vertex, and edge contributions.  We compute the resulting (unreduced and trace-free) virtual dimensions and analyze the interaction terms relevant to a possible reduction to rank one.  We emphasize that, on a compact toric threefold, colored partitions do not classify all torus-fixed stable higher-rank sheaves: the resulting series is therefore a proposed abelianized/framed vertex series unless an additional comparison theorem with the full stable-sheaf moduli problem is supplied.

When the threefold obstruction theory is not available in the desired form, we pass to the local Calabi--Yau fourfold $Y=\operatorname{Tot}(K_X)$ and formulate the corresponding colored DT4 vertex.  A $K_X$-twisted tautological insertion removes fixed points with nontrivial fiber layers, reducing solid partitions to plane partitions at the level of the local vertex.  In rank one we give a scheme-theoretic explanation: the Hilbert functor of subschemes of $X$ is canonically isomorphic to the constrained Hilbert functor of subschemes of $Y$ lying scheme-theoretically in the fixed zero section.  Under the hypotheses of the DT4 K-theoretic Lefschetz/virtual-pullback formalism, the associated support insertion reduces the rank-one fourfold theory to the threefold theory.  We keep this global rank-one statement logically distinct from the higher-rank colored-vertex reduction.
\end{abstract}
\tableofcontents

\section{Introduction}

Let $X$ be a smooth projective toric threefold, with K\"ahler form $\omega$ and canonical bundle $K_X$.  Since $X$ is compact and toric, it cannot be Calabi--Yau.  When $H^0(X,K_X^{-1})\neq0$ (for example when $X$ is Fano), moduli spaces of stable sheaves admit the standard threefold deformation theory under the usual hypotheses.  It is natural to ask whether suitably inserted higher-rank invariants are determined by rank-one data.  In the Calabi--Yau setting such rank-reduction statements are known in several regimes \cite{FeyzbakhshThomasRank0:2024,Feyzbakhsh:2021nds,DelZotto:2021gzy,FasolaMonavariRicolfi:2021}; recent work also advances the Fano case \cite{KarpovMoreira:2026}.

The vertex formalism used below should, however, be distinguished from the full compact stable-sheaf problem.  A colored partition at every toric chart describes the Coulomb-abelianized/framed sector in which the fixed object splits equivariantly into rank-one constituents.  General torus-fixed stable sheaves on a compact toric threefold are described by Klyachko-type data and their fixed loci need not be isolated; already in rank two, fixed components with constant reflexive hull can be products of projective lines \cite{GholampourKoolYoung:2018}.  Consequently, our threefold series is a proposed abelianized vertex series.  It becomes the full higher-rank Donaldson--Thomas series only in a moduli problem for which the colored-partition fixed locus is exhaustive, or after a separate comparison theorem.

A second route is to work on the noncompact Calabi--Yau fourfold $Y=\operatorname{Tot}(K_X)$ and study dimensional reduction to the zero section.  This is especially useful when one wants a DT4 formulation with tautological insertions.  The local colored vertex admits a direct fiber-reduction mechanism.  Globally, the fixed-divisor Hilbert-functor argument in Section~4.4 is rank one; we state this explicitly and do not use it as a proof of a general rank-$r$ stable-sheaf correspondence.

\subsection{Equations}

The fourfold self-duality equations \cite{Corrigan:1982th, Donaldson:1996kp}
constrain the $(0,2)$ part of the curvature of a connection on
a principal $U(r)$ bundle on $Y$ 
\begin{equation}
F = \bar \star_4 F
\end{equation}
where self-duality uses the $(4,0)$ holomorphic form of $Y$.  We impose stability
\begin{equation}
F^{1,1} \wedge \omega^3 = \omega^4 \mu(F) \cdot 1
\end{equation}
where the slope $\mu(F)$ depends on $c_1(F)$.
Let $\cm$ be (some compactification of) the space of solutions to
these equations, modulo gauge equivalence.  It is a disjoint union over
topological sectors.

The corresponding ADHM-type model \cite{Nekrasov:2018xsb}
involves four square matrices $B_a \in \End(\BC^k)$
and one rectangular matrix $I \in \Hom (\BC^r,\BC^k)$, satisfying
\begin{equation}
\begin{aligned}
[B_a,B_b] + \epsilon_{abcd} [B_c,B_d]^\dagger &= 0, \quad a,b,c,d=1,2,3,4\\
\sum_{a=1}^4 [B_a, B_a^\dagger] + II^\dagger &= \zeta \cdot 1_k
\end{aligned}
\end{equation}
modulo $U(k)$ gauge equivalence.  Here $\zeta$ is a real stability parameter,
and $\epsilon$ the completely antisymmetric tensor in four indices,
determined by the $SU(4)$ structure on $Y$.
The case $r=1$ corresponds to $\Hilb^k (\BC^4)$.

The 3d equations on $X$ also involve an adjoint-valued (3,0) form $\varphi$
\begin{equation}
\begin{aligned}
F_A^{2,0} + \partial_A^\dagger \varphi &=0 \\
F_A^{1,1} \omega^2 + [\varphi,\bar\varphi] &= \mu(F) \omega^3 \cdot 1
\end{aligned}
\end{equation}
Let $\cn$ be a compactification of the space of solutions modulo gauge equivalence,
again a disjoint union over topological data.

The corresponding ADHM-like model for $\BC^3$ is conveniently written with four $k\times k$ matrices $B_1,B_2,B_3,B_4$ and one rectangular matrix $I$.  The first three matrices encode the three spatial complex directions, while $B_4$ encodes the adjoint Higgs field.  They satisfy
\begin{equation}
\begin{aligned}
[B_a,B_b] + \epsilon_{abc4} [B_c,B_4]^\dagger &= 0, \quad a,b,c=1,2,3\\
B_4 I &=0 \\
\sum_{a=1}^4 [B_a, B_a^\dagger] + II^\dagger &= \zeta \cdot 1_k
\end{aligned}
\end{equation}
modulo the gauge symmetry.

\section{Geometry}

Within this section, let $X$ be a toric variety of complex dimension $d$.

If $X$ is non-compact (such as the Calabi-Yau case), one needs
to make further assumptions \cite[Thm 6.7]{karshon-lerman:2015}
for it to also be a symplectic quotient.
The case of the total space of the canonical bundle over a Fano threefold
is admissible, but this also works for the total space of the canonical bundle of
$F_3 \times \BP^1$, the product of a Hirzebruch surface and
projective line, or more general extensions thereof.

Denote the Newton polyhedron of $X$ by $\Delta=\Delta(X)$, and the local equivariant
weights near a fixed point $v\in\Delta_0$ by $q_1^v, \dots,q_d^v$.
By this we mean the generators of the equivariant K-theory of a point,
such that the tangent space to $X$ at $v$ has character $q_1 + \dots +q_d$.
We usually drop the label $v$.

Denote by $\psi$ the transition functions of an edge,
namely if direction 1 is tangent to the edge, then along the edge we have
$q_1 \to q_1^{-1}$ and $q_p \to q_p q_1^{\psi_p}$ for $p \neq 1$.
For a Calabi--Yau $d$-fold, preservation of $Q=\prod_a q_a$ under this transition gives
\begin{equation}
 \sum_{p\neq1}\psi_p=2.
\end{equation}
Notice that this is independent of $d$; the value $d-1$ happens to agree with it only in dimension three.

For any $A \subseteq \{1,\dots,d\}$, denote the products by $P_A \coloneq
	\prod_{a \in A} (1-q_a)$, $Q_A \coloneq \prod_{a \in A} q_a$,
as well as $P \coloneq P_{\{1,\dots,d\}}$ and $Q \coloneq Q_{\{1,\dots,d\}}$.
The CY condition forces $Q=1$.
The dual of $q_i$ is $q_i^*=q_i^{-1}$, and similarly for
Laurent polynomials.

Given a Laurent polynomial in the group-valued equivariant
parameters with integer coefficients $c_i$ and without
the constant term, define its plethystic exponential as
\begin{equation} \label{eq:ahat}
	\hat a :
	\sum_{i\in I} c_i m_i \mapsto
	\prod_{i \in I} (m_i^{1/2} - m_i^{-1/2})^{-c_i}
\end{equation}
Note that $\hat a(0)=1$.
Given a Laurent polynomial $\cp(x)$ in variables $x$ with integer coefficients,
define its size as $|\cp| = \lim_{x \to 1} \cp(x)$.

\subsection{Example: complex projective space}

Denote the homogeneous coordinates of $\BC^4$ as $X_i$, and momenta as
$p_i = |X_i|^2$, with the big torus acting as $X_i \mapsto e^{i \ep_i}
	X_i$, while the small torus $T$ as $X_i \mapsto e^{i \theta} X_i$.
The momentum map is
\begin{equation}
	\mu(p) \coloneq p_1 + p_2 + p_3 + p_4
\end{equation}
and the symplectic quotient is $X = \mu^{-1}(t)/T$ for some regular
value $t>0$ of $\mu$.  This is equivalent to a GIT quotient with $\BC^4
	\setminus \{ 0\}$ as semistable locus.

We have for the local coordinates
\begin{equation}
	\begin{array}{ccc}
		X_i \neq 0 & \text{weights}            & H       \\
		\midrule
		1          & X_2/X_1, X_3/X_1, X_4/X_1 & \ep_1 t \\
		2          & X_1/X_2, X_3/X_2, X_4/X_2 & \ep_2 t \\
		3          & X_1/X_3, X_2/X_3, X_4/X_3 & \ep_3 t \\
		4          & X_1/X_4, X_2/X_4, X_3/X_4 & \ep_4 t
	\end{array}
\end{equation}
where $H$ is the Hamiltonian of the torus action.  From this table we
can read off the local equivariant weights $q_a$ in terms of the global
ones $T_i = \exp i \ep_i$.  The polyhedron $\Delta$ is a tetrahedron,
and the Euler characteristic of $X$ is $\chi=4$.

\section{Threefold theory}
\subsection{Universal objects}

\paragraph{Scope of the colored fixed-point model.}
Throughout Sections~3.1--3.6 we work in the abelianized/framed sector in which a torus-fixed object is represented equivariantly by a direct sum of rank-one constituents
\[
 F_{\mathrm{ab}}\simeq\bigoplus_{i=1}^r L_i\otimes I_i,
\]
with independent Coulomb weights.  The corresponding local fixed data are $r$ colored asymptotic plane partitions.  This is the natural fixed-point description for framed/Quot-type vertex problems, but it is not a classification of the torus-fixed locus of the ordinary compact moduli space of stable rank-$r$ sheaves.  We therefore write the formulas below as characters of this colored sector and only identify them with a full DT theory when the chosen moduli problem has this fixed locus.

The Grothendieck--Riemann--Roch theorem gives the unreduced deformation character at a colored object $F$
(the minus sign is due to our convention for $\hat a$)
\begin{equation}\label{eq:threefold-unreduced-character}
 T_{\mathrm{unred}}=-\int_X\ch(E|_F\otimes E|_F^{\vee})\td(X).
\end{equation}
For a simple sheaf in the usual fixed-determinant DT problem one must remove the scalar trace.  Thus the trace-free character is
\begin{equation}\label{eq:threefold-tracefree-character}
 T_0=T_{\mathrm{unred}}+\chi(\mathcal O_X).
\end{equation}
For a smooth projective toric variety $\chi(\mathcal O_X)=1$.  This trace correction should not be confused with the separate ``one empty color per summand'' normalization used later in the abelianized vertex series.

Let $E$ denote the universal object whenever such a universal family exists on the chosen moduli presentation, and let $\pi_1,\pi_2$ be the two projections.  For every vertex and $i=1,\dots,r$, let $\tK_i$ be an asymptotic plane partition.  The profiles at the two ends of every edge are required to match so that the local monomial data glue.  We identify a partition with its character, introduce Coulomb parameters $\ta_i$, and write $\ta_{ij}=\ta_i/\ta_j$.

From localization the unreduced colored character is
\begin{equation}
 T_{\mathrm{unred}}=-\sum_{v\in\Delta_0}\frac{\ch(E_v)\ch(E_v^{\vee})}{P^*}.
\end{equation}
where the pullback of the universal object to a fixed point has character
\begin{equation}
 \ch(E_v)=\sum_{i=1}^r\ta_i(1-P\tK_i).
\end{equation}
For notational economy we write $T$ for $T_{\mathrm{unred}}$ in the local redistribution formulas below; the trace-free correction \eqref{eq:threefold-tracefree-character} is restored when comparing with the ordinary simple-sheaf obstruction theory.

\subsection{Redistribution with magnetic fluxes}

Substitution gives
\begin{equation}
	T = \sum_v \sum_{i,j} \ta_{ij}
	\left( \tK_j^* - Q \tK_i - P \tK_i \tK^*_j - \frac{1}{P^*} \right)
\end{equation}
The asymptotic partitions $\tK_i$ can be infinite both along curves and
divisors in $X$.  To deal with the latter, observe that our
sheaf $F$ may have (equivariant) first Chern class supported on compact
toric divisors in $X$: this gives non-negative local asymptotic multiplicities (magnetic fluxes in physics)
\begin{equation}
	\tn_{i,f} \coloneq \lim_{q_a, q_b \to 1} P_{ab} \tK_i,
	\quad f \in \Delta_2,
	\quad i=1,\dots,r
\end{equation}
(a face is locally spanned by two directions, here $a,b$).  Again, we
identify partitions with their characters, as well as their size in 1d.
Subject to the toric divisor relations these local multiplicities determine an equivariant line bundle $L_i$; its divisor class need not have a unique effective presentation.  We write $\ch(L_i)=q^{\tn_i}$.  Twisting by such line bundles, define
\begin{equation}
	K_i \coloneq q^{-\tn_i} \left( \tK_i - \frac{1-q^{\tn_i}}{P} \right)
\end{equation}
as well as shifted Coulomb parameters
\begin{equation}
	a_i \coloneq \ta_i q^{\tn_i}
\end{equation}
Now $K$s can only be infinite along curves, and we can write
\begin{equation}
	T = \sum_v \sum_{i,j} a_{ij}
	\left( K_j^* - Q K_i - P K_i K^*_j - \frac{1}{P^*} \right)
\end{equation}
in terms of the new variables.
For convenience, we define the perturbative part
\begin{equation}
	T_{pert} \coloneq - \sum_v \sum_{i,j} \frac{a_{ij}}{P^*}
\end{equation}
whose coefficient of each Coulomb monomial $\ta_{ij}$ depends only on the fluxes.  More precisely it is a finite sum of equivariant Euler characters of the line bundles $L_i\otimes L_j^*$.  It is therefore a Laurent polynomial because $X$ is compact.  The rest we call the instanton part, so that
\begin{equation}
	T = T_{pert} + T_{inst}
\end{equation}
and we redistribute it \cite{Maulik:2003rzb}.
For each edge, introduce 2d partitions
\begin{equation}
	\la_{i,e} \coloneq \lim_{q_a \to 1} (1-q_a) K_i,
	\quad e \in \Delta_1
\end{equation}
where edge corresponds to direction $a$, and regularized partitions
\begin{equation}
	K_{i,reg} \coloneq
	K_i - \sum_{e \in \Delta_1 | v \in e} \frac{\la_{i,e}} {P_e}
\end{equation}
where the sum runs over all edges that meet at a given vertex,
and we identify $e$ with its local direction at $v$.
Then define the local edge contribution
\begin{equation} \label{eq:teij}
	t_{e,ij} \coloneq \frac{\la^*_{j,e}}{P_e^*} - Q \frac{\la_{i,e}}{P_e}
	- P \frac{\la_{i,e} \la^*_{j,e}}{P_e P_e^*},
	\quad e \in \Delta_1
\end{equation}
and the vertex contribution $T_{v,ij}$, $v \in \Delta_0$, such that
\begin{equation}
	K_j^* - Q K_i - P K_i K^*_j = T_{v,ij} + \sum_{e | v \in e} t_{e,ij}
\end{equation}
Note that $T_{v,ij}$ can be expressed in terms of $K_{i,reg}$
and $\lambda_i$ only, thus it is manifestly finite:
\begin{multline} \label{eq:Tvij}
	T_{v,ij} = K_{j,reg,v}^* - Q K_{i,reg,v} - P K_{i,reg,v} K^*_{j,reg,v} \\
	-P K_{i,reg,v} \sum_{e | v \in e} \frac {\lambda^*_{j,e}} {P_e^*}
	-P K_{j,reg,v}^* \sum_{e | v \in e} \frac {\lambda_{i,e}} {P_e}
	-P \sum_{e \neq e'} \frac {\lambda_{i,e} \lambda^*_{j,e'}} {P_e P_{e'}^*}
\end{multline}
If we let
\begin{equation}
	T_{e,ij} \coloneq \sum_{v \in e} q^{\tn_{ij}} t_{e,ij}
\end{equation}
then we can show that it is also a polynomial, and we get
\begin{equation} \label{eq:tinst}
	T_{inst} = \sum_{i,j} \ta_{ij}
	\left(
	\sum_{v \in \Delta_0} q^{\tn_{ij}} T_{v,ij} + \sum_{e \in \Delta_1} T_{e,ij}
	\right)
\end{equation}

The redistributed character is a Laurent polynomial in
$\BZ[T_i^\pm,\ta_i^\pm]$.  The operation $\hat a$ in
\eqref{eq:ahat}, however, is defined only after removing zero-weight
summands.  Accordingly, throughout the localization formula we apply
$\hat a$ to the moving/nonconstant part of the character (or,
equivalently, first divide out the fixed empty/trace contribution in the
chosen normalization), and then sum over colored partitions in the fixed
topological sector.

\subsection{Weights}

To relate the topological weights to the Chern character of $F$, introduce
the bookkeeping expression
\begin{equation}
I \coloneq \frac{1}{g_s^2}\int_X e^\omega\cdot\ch(g_sF).
\end{equation}
Here the notation is shorthand for the degree rescaling
\[
 \ch(g_sF):=\sum_{k=0}^3g_s^k\ch_k(F),
\]
not for tensoring $F$ by a scalar.  In a fixed rank and first-Chern-class
sector, the $g_s^{-2}$ and $g_s^{-1}$ pieces are common topological factors.
After suppressing those common factors, equivariant localization gives the
partition-dependent terms
\begin{equation}
I_3=-\sum_{v,i}|K_{i,v,reg}|-\sum_{e,i}f_e(\lambda_{e,i})
-\sum_{i,e}|\lambda_{i,e}|\int_e c_1(L_i),
\end{equation}
\begin{equation}
I_2=-\sum_{i,e}t_e|\lambda_{i,e}|,
\qquad t_e:=\int_e\omega,
\end{equation}
so that the partition-dependent part of $I$ is
$g_sI_3+I_2$.  Here $f_e(\lambda)$ denotes the finite edge/framing
contribution to the degree-three Chern character obtained from the edge
transition formula \eqref{eq:edge-line-bundle-sum}; it is determined by the
transition degrees $\psi$ and the partition $\lambda$.  This convention
also removes the ambiguous $O(\tn)$ notation: all remaining terms depend
only on the fixed rank/flux sector and contribute a common topological
prefactor to the generating function.

\subsection{Virtual dimension}

For the instanton part of the colored unreduced character we obtain
$|T_{inst}|=\sum_{e\in\Delta_1}\vdim_e$, where
\begin{equation}\label{eq:vdime}
 \vdim_e=
 r\sum_i |\lambda_{i,e}|\int_e c_1(T_X).
\end{equation}
Although magnetic fluxes enter the equivariant edge characters, their
linear contribution to the size cancels between the ordered color pairs.
This agrees with Riemann--Roch: for
$F_{\mathrm{ab}}=\bigoplus_i L_i\otimes I_i$, the endomorphism character
$\ch(F_{\mathrm{ab}}^\vee\!\otimes F_{\mathrm{ab}})$ is self-dual, so its
odd Chern-character part---and in particular a term linear in
$c_1(L_i-L_j)$ times a curve class---cancels after summing over $i,j$.
The flux dependence remains in the perturbative Euler characteristics.

For the perturbative part,
\begin{equation}\label{eq:pert-size}
 |T_{pert}|=-\sum_{i,j}\chi(X,L_i\otimes L_j^*).
\end{equation}
Hence the unreduced colored dimension is $|T_{inst}|+|T_{pert}|$.  If the
character comes from the ordinary moduli of simple sheaves with fixed
determinant, the scalar trace is removed and the virtual dimension is
instead
\begin{equation}\label{eq:tracefree-vdim}
 \vdim_{\mathrm{DT}}=|T_{inst}|+|T_{pert}|+\chi(\mathcal O_X)
 =|T_{inst}|+|T_{pert}|+1
\end{equation}
for a smooth projective toric threefold.  In particular, for rank one,
zero flux and point-like data on a Fano threefold, this gives virtual
dimension zero, as expected.

\subsubsection{Proof} \label{sec:vdim.proof}

Let us compute $|T_{inst}|$, as defined in \zcref{eq:tinst}.
The first term does not contribute, since we have $|q^{\tn_{ii}} T_{v,ii}|=0$
and $|q^{\tn_{ij}} T_{v,ij} + q^{\tn_{ji}} T_{v,ji}|=0$ for $j \neq i$.
This is special to dimension 3 (or odd).
So we get $|T_{inst}| =  \sum_{i,j} \sum_e |T_{e,ij}|$.
Let us discuss each edge separately.
For $i=j$, from \zcref{eq:teij} we get
\begin{equation} \label{eq:Tii}
|T_{e,ii}|
= |\sum_{v \in e} t_{e,ii}|
= |\sum_{v \in e} \frac{\lambda_{i,e}}{P_e} (1-Q)|
= |\lambda_{i,e}| \int_e c_1 (T_X)
\end{equation}
where we took the conjugate of half the terms and used
the claim that quadratic terms do not contribute.
For $i \neq j$, we get
\begin{multline}
|T_{e,ij}| + |T_{e,ji}^*|
= |\sum_{v \in e} q^{\tn_{ij}} (t_{e,ij} - Q^* t_{e,ij})| \\
= |\sum_{v \in e} q^{\tn_{ij}} \frac {\lambda_{i,e}}{P_e} (1-Q)
+ q^{\tn_{ij}} \frac {\lambda^*_{j,e}}{P_e^*} (1-Q^*)|
\end{multline}
the last equality again following from the claim.
The first term gives
\begin{equation}
\left|\sum_{v\in e}q^{\tn_{ij}}\frac{\lambda_{i,e}}{P_e}(1-Q)\right|
=|\lambda_{i,e}|\int_e c_1(T_X),
\end{equation}
and the second gives
\begin{equation}
\left|\sum_{v\in e}q^{\tn_{ij}}\frac{\lambda_{j,e}^*}{P_e^*}(1-Q^*)\right|
=|\lambda_{j,e}|\int_e c_1(T_X).
\end{equation}
Indeed, the line-bundle transition changes the equivariant polynomial but
not its size: after the two endpoint contributions on the compact edge are
combined, its character tends to $1$ in the nonequivariant limit.  Equivalently,
the possible terms linear in $c_1(L_{ij})$ cancel in the full ordered-color
sum, as predicted by the self-duality of
$\ch(F_{\mathrm{ab}}^\vee\!\otimes F_{\mathrm{ab}})$.
Consequently,
\begin{equation}
\sum_{i<j}\bigl(|T_{e,ij}|+|T_{e,ji}^*|\bigr)
=(r-1)\sum_i|\lambda_{i,e}|\int_e c_1(T_X).
\end{equation}
Combining this with $\sum_i |T_{e,ii}|$ from \zcref{eq:Tii} gives
\zcref{eq:vdime}.

Back to the claim: let us show first that the equality
\begin{equation}
|\sum_{v \in e} P \frac{\lambda_{i,e}}{P_e} \frac{\lambda^*_{i,e}}{P^*_e}|=0
\end{equation}
holds.  Let us pick coordinates where $e$ has tangent direction 1,
and drop the label $e$: we get for the LHS
\begin{equation}
\sum_{v \in e} \frac{P_{23}} {P^*_1} \lambda_i \lambda^*_i
= \sum_{\Box=(a,b),\Box'=(a',b') \in \lambda_i}
\sum_{v \in e} \frac{P_{23}} {P^*_1} q_2^{a-a'} q_3^{b-b'}
\end{equation}
We can discuss each summand separately: we get
\begin{multline}
\sum_{v \in e} q_2^{a-a'} q_3^{b-b'} \frac{P_{23}} {P^*_1}
= q_2^{\Delta_a} q_3^{\Delta_b}
\Big(
\frac{1-q_1^{\psi \cdot \Delta-1}}{1-q_1^{-1}}
-q_2 \frac{1-q_1^{\psi_2 + \psi \cdot \Delta-1}}{1-q_1^{-1}} \\
-q_3 \frac{1-q_1^{\psi_3 + \psi \cdot \Delta-1}}{1-q_1^{-1}}
+q_2 q_3 \frac{1-q_1^{\psi_2+\psi_3 + \psi \cdot \Delta-1}}{1-q_1^{-1}}
\Big)
\end{multline}
where $\Delta_a = a-a'$, $\Delta_b = b-b'$,
and $\psi \cdot \Delta = \psi_2 \Delta_a + \psi_3 \Delta_b$.
The term in parenthesis is the sum of four polynomials,
and it vanishes in the limit $q_{1,2,3} \to 1$.

Let us repeat the same argument for the case $i \neq j$.
Let us pick coordinates where $e$ has tangent direction 1,
and drop the label $e$, and let $\tn=\tn_{ij}$.  We get
\begin{equation}
\sum_{v \in e} q^\tn \frac{P_{23}} {P^*_1} \lambda_i \lambda^*_j
= \sum_{\Box=(a,b) \in \lambda_i, \Box'=(a',b') \in \lambda_j}
\sum_{v \in e} \frac{P_{23}} {P^*_1} q^\tn q_2^{a-a'} q_3^{b-b'}
\end{equation}
We can discuss each summand separately.
We denote by $\tn_{1,2,3}$ the three components of $\tn$ at $v$,
and by $\tn'$ the components at $v' \in e$
(2 and 3 are the same, while 1 can be different).  We get
\begin{multline}
\sum_{v \in e} q_2^{a-a'+\tn_2} q_3^{b-b'+\tn_3}
q_1^{\tn_1(v)} \frac{P_{23}} {P^*_1}
= q_2^{\Delta_a} q_3^{\Delta_b} q_1^{\tn_1}
\Big(
\frac{1-q_1^{\rho}}{1-q_1^{-1}}
-q_2 \frac{1-q_1^{\psi_2 + \rho}}{1-q_1^{-1}} \\
-q_3 \frac{1-q_1^{\psi_3 + \rho}}{1-q_1^{-1}}
+q_2 q_3 \frac{1-q_1^{\psi_2+\psi_3 + \rho}}{1-q_1^{-1}}
\Big)
\end{multline}
where $\Delta_a = a-a'+\tn_2$, $\Delta_b = b-b'+\tn_3$,
$\psi \cdot \Delta = \psi_2 \Delta_a + \psi_3 \Delta_b$,
and $\rho = \psi \cdot \Delta-1-\tn_1'-\tn_1$.
The term in parenthesis is the sum of four polynomials,
and it vanishes in the limit $q_{1,2,3} \to 1$.
This shows that
\begin{equation}
|\sum_{v \in e} q^{\tn_{ij}} \frac{P_{23}} {P^*_1}
\lambda_i \lambda^*_j| =0
\end{equation}

\subsection{Generating functions}

We want to compute generating functions of
\begin{equation}
\chi (\cn, \hat O^{vir} \otimes \cL)
\end{equation}
where hat means twist by the square
root of the canonical line bundle,
$O^{vir}$ is the virtual structure sheaf,
and $\cL$ is a K-theory class on $\cn$.
The role of $\cL$ is
to saturate the virtual dimension: while it can be pulled
back from a class on $X$, interesting observables
are those coming from $\cn$, such as the pushforward via $\pi_1$
of the universal sheaf
and its exterior and symmetric powers.

We choose to weigh each summand by $I$,
and take the plethystic $\hat a (T_{inst})$, since all expressions
are finite:
\begin{equation}
	z_{inst} (\ta, \tn) \coloneq
	\sum_{\lambda} e^{I_2} \sum_{K | \lambda} e^{g_s I_3} \,
	\hat a (T_{inst}) \cdot \cL (K,\lambda,\tn)
\end{equation}
where the second sum is over all $K$s with boundary $\lambda$s.  In the colored vertex normalization it is convenient to remove the empty contribution of each diagonal color and set
\begin{equation}\label{eq:abelian-pert-normalization}
 z_{pert}^{\mathrm{ab}}\coloneq \hat a(T_{pert}+r).
\end{equation}
The $+r$ in \eqref{eq:abelian-pert-normalization} is an abelianized normalization, not the scalar trace correction of the ordinary rank-$r$ stable-sheaf deformation theory.  For the latter the correction is $+\chi(\mathcal O_X)=+1$, cf.\ \eqref{eq:threefold-tracefree-character}.

We therefore define the proposed colored/abelianized vertex series
\begin{equation}\label{eq:abelianized-series}
 Z_X^{\mathrm{ab}}(\cL)=\sum_{\text{admissible }\tn}
 z_{pert}^{\mathrm{ab}}\,z_{inst}.
\end{equation}
Compactness of $X$ makes each fixed topological-sector character a Laurent polynomial after the vertex/edge redistribution.  Independence of the Coulomb parameters $\ta$ is an additional rigidity/factorization statement and should not be inferred from compactness alone.  Likewise, \eqref{eq:abelianized-series} agrees with the K-theoretic DT generating function of ordinary stable rank-$r$ sheaves only when a comparison with the full torus-fixed moduli problem has been established.

\subsection{Interaction terms and square roots}

Both the vertex and edge terms contain interaction terms with $i \neq
j$, hence it is not obvious that $\ta$ decouple (certainly it does not
decouple for the perturbative part): we need to check whether $z_{inst}$
depends on $\ta$ or not.  If the interaction term is independent
of $\ta$, then we may get a factorization similar to the CY case
\cite{DelZotto:2021gzy}, with extra contributions to the topological
terms.  At the same time, we need to be able to take square roots, and we
will see that the two issues are related:  for the non-CY case, one option
\cite{Nekrasov:2014nea} is to twist the virtual canonical bundle $K_{vir}$
by $\det R\pi_{1,*}(E\otimes(L_4-L_5))^{\otimes r}$, where the two line bundles on $X$ satisfy $L_4\otimes L_5=K_X$, and then take a square root.

\subsubsection{Points}

Let us discuss points first, i.e.\ the case of $\BC^3$.  In this case
we can take $L_4=L_5=O$. The contribution to $K_{vir}$ only comes from
the determinant of $K^*-QK$ in \zcref{eq:Tvij}, and $K_{vir} = \prod_i Q
x_i^2$, with $x_i$ being the content of $K$: this is manifestly a square.

In the limit $\ta\to\infty$, the vertex only contributes terms
\begin{equation}
	\prod_v \prod_i (-Q^{1/2})^{(r+1-2i)|K_{i,v,reg}|}
\end{equation}
that we can reabsorb in the coupling.

\subsubsection{Curves}

For $m \in \BZ$, consider the Laurent polynomial
$\cp_m (q_1) = \frac{1-q_1^m}{1-q_1}$.
For any line bundle $L$ on $X$, we then have
\begin{equation}\label{eq:edge-line-bundle-sum}
	\sum_{v \in e} \ch(L)\frac {\lambda}{P}
	= \sum_{\Box \in\lambda} m_\Box \cp_{1+\psi.a+\psi_L}
\end{equation}
where $\psi$ are transition functions, and
$m_\Box = q_2^{a-1} q_3^{b-1}$ is the monomial
corresponding to box $(a,b) \in \lambda$.

For the edge character in a toric threefold, from \zcref{eq:teij} we get
\begin{equation}
	\sum_{v \in e} t_{e,ii}
	= \sum_{\Box \in \lambda} ( m_{\Box}^* \cp_{1+\psi.a}^*
	- Q m_{\Box} \cp_{1+\psi.a+(\psi_2+\psi_3-2)} )
\end{equation}
up to overall terms from the quadratic part.
While for a CY the determinant of the above equation is a perfect square,
if $Q\neq1$ we need to correct it \cite{Nekrasov:2014nea}.  Write $q_4,q_5$ for the local characters of $L_4,L_5$ and $\psi_4,\psi_5$ for their transition degrees along the edge; then $q_4q_5=Q^{-1}$ because $L_4\otimes L_5=K_X$.  After taking into account the pushforward, the corrected determinant takes the form
\begin{equation}
K_{vir}\otimes
\Bigl(\det R\pi_{1,*}(E\otimes(L_4-L_5))\Bigr)^{\otimes r}
= \prod_e \prod_{\Box \in \lambda_e}
(q_1^{-\psi_5}q_4m_\Box^{-1})^{2(1+\psi.a)}.
\end{equation}
The right-hand side is a perfect square, up to the constant factors raised
to the power $|\lambda|$.

Likewise, the large $\ta$ limit
\begin{equation}
	T_{e,ij} + T_{e,ji}^* = \sum_{v \in e} q^{\tn_{ij}} (t_{e,ij}-Q^* t_{e,ij})
\end{equation}
produces terms of the form
\begin{equation}
	\sum_{(a',b') \in \lambda_j}
	\sum_{(a,b) \in \lambda_i} q_2^{\Delta_2} q_3^{\Delta_3}
	(\cp_\rho - Q^* \cp_{\rho-(\psi_2+\psi_3-2)})
\end{equation}
with notation as in \zcref{sec:vdim.proof}.

\section{Fourfold theory}

Consider a toric Calabi--Yau fourfold $Y$ and let $\cm=\bigsqcup_\gamma\cm_\gamma$ denote a moduli space of ideal sheaves with Chern character $\gamma=(1,c_1,*,*,*)$.  We assume the relevant DT4 orientation data exist, so that the twisted virtual structure sheaf $\widehat{\mathcal O}^{\mathrm{vir}}$ is defined in the sense of Oh--Thomas \cite{Oh:2020rnj}.  With a line bundle $L_\mu$ on $Y$ we consider the localized equivariant K-theoretic index
\begin{equation}\label{eq:fourfold-index}
 Z_\gamma=
 \chi_T\!\left(\cm_\gamma,
 \widehat{\mathcal O}^{\mathrm{vir}}\otimes
 \widehat\Lambda^{\bullet} R\pi_{1*}(E\otimes\pi_2^*L_\mu)\right),
\end{equation}
which takes values in the localized representation ring $K_T(\mathrm{pt})_{\mathrm{loc}}$, not in $K_T(\cm_\gamma)$.  Since $\cm_\gamma$ can be noncompact, the non-equivariant limit need not exist.

The superscript ``$2$'' used below on a self-dual character is notation for the full symmetric DT4 character; it is not an ordinary square.  A ``square root'' or half-character $T$ means a choice satisfying $T+T^*=T^{(2)}$, together with the compatible orientation sign.

We assume that torus-fixed components of $\cm$ consist of isolated fixed
points, so that the vertex formalism works.  The advantage of the vertex
formalism is that it is universal, in the sense that the same building
blocks can be assembled to produce invariants for different choices of
$Y$, without the need to recompute everything from scratch.

First observe that a torus-fixed ideal sheaf $F=I_Z\in\cm$ is determined by a torus-invariant subscheme $Z$; it is the quotient $\mathcal O_Z$, rather than the rank-one ideal sheaf itself, whose support lies on fixed points, curves, surfaces and divisors.  Near each fixed
point $v \in \Delta_0 (Y)$, we have a chart where $Y$ is isomorphic to
$\BC^4$, and $F$ looks like a solid partition.  This solid partition can
be infinite along curves, surfaces and threefolds that meet at $v$, and it
must be \emph{compatible} with other solid partitions associated to $F$
in different charts.  Let $(\tK_v)_v$ be such a collection of compatible
solid partitions.  Its topological invariants, such
as sizes etc., are determined by $\gamma=\ch(F)$.

Having classified the fixed points, localization requires knowledge of 
the virtual tangent at a fixed point, $- \chi (F,F)_0$, where
\begin{equation}
\chi (F,F) = \sum_i (-1)^i \Ext^i (F,F)
\end{equation}
and subscript means removal of the constant part.  While the computation
of $\chi$ is standard across dimensions, in complex dimension four the
new challenge is the need to choose a half-character (square root) compatible with the DT4 Serre-duality pairing and an orientation.

The application of GRR theorem gives
\begin{equation}
\chi(F,F) = \int_Y \ch(E|_F \otimes E^*|_F) \td(Y)
\end{equation}
where $E^*$ denotes the dual of $E$, and we restricted it to $F \in \cm$.
Localization theorem then gives
\begin{equation}
RHS = \sum_{v \in \Delta_0} \frac{\ch(E_v) \ch(E_v^*)}{P_{1234}}
\end{equation}
where $E_v$ is pulled back via the inclusion of fixed point $v \in Y$,
and its virtual character is
\begin{equation}
\ch(E_v) = 1-P_{1234} \tK_v
\end{equation}
The constant (perturbative) part corresponds to the empty partition
\begin{equation}
\chi(O_Y,O_Y) = \sum_{v\in\Delta_0} \frac{1}{P_{1234}^*}
\end{equation}
and it is a rational function since $Y$ is non-compact.

Our choice of insertion, being written in terms of the universal
sheaf, has the nice property that
it can be taken into account by replacing $\ch(E_v)$ with
\begin{equation}
\ch(E_v) - \ch(L_\mu)
 = 1 - \mu - P_{1234} \tK
\end{equation}
in the GRR formula, where $\mu$ is the character of $L_\mu$.
Locally at $v$, define the self-dual character
\begin{equation}\label{eq:rankone-full-dt4-character}
T^{(2)} \coloneq (1-\mu)^* \tK + (1-\mu) \tK^* - P_{1234} \tK \tK^*.
\end{equation}
Its $\mu$-independent part satisfies
$\sum_{v\in\Delta_0}T^{(2)}|_{\mu=0}=-\chi(F,F)_0$.

If $\xi = c_1 (F) \neq 0$, we should get rid of it by twisting: let $\tn$
the equivariant first Chern class, and introduce the corresponding line
bundle $L_\tn$.  If we define
\begin{equation}
E_{twist} = L_\tn^* \otimes E
\end{equation}
then it satisfies
\begin{equation}
E_{twist} \otimes E_{twist}^* = E \otimes E^*
\end{equation}
and we end up with a collection of compatible solid partitions
\begin{equation}
K = q^{-\tn} (\tK - \frac{1-q^\tn}{P_{1234}}) 
\end{equation}
that can be infinite along curves and surfaces only,
and a shift in $\mu$.

A local half-character is
\begin{equation}\label{eq:rankone-half-character}
T=(1-\mu)^*K-P_{123}KK^*,
\end{equation}
and the Calabi--Yau relation implies
$T+T^*=T^{(2)}$, since $P_{123}+P_{123}^*=P_{1234}=P_{1234}^*$.

Assuming isolated fixed points and a compatible orientation, DT4 localization gives
\begin{equation}\label{eq:rankone-dt4-localization}
Z_\gamma = \sum_{\operatorname{inv}(\{\tK\})=\gamma} (-1)^{s(\{\tK\})}
\hat a \left( \sum_{v\in\Delta_0} T_v \right).
\end{equation}
The orientation signs are additional data.  A concrete rank-one vertex sign prescription is given in \cite{Nekrasov:2023nai}; for an assembled noncompact geometry, however, the global edge/face interaction signs and compatibility are part of the conjectural vertex formalism rather than a general theorem.  If we let the generating sum in a given sector $\xi$
\begin{equation}
Z_\xi \coloneq \sum_{\gamma | c_1=\xi} Z_\gamma \cdot w(\gamma)
\end{equation}
where we make some choice of weight $w$,
then we get
\begin{equation}
Z_\xi = \sum_{\{\pi,\lambda\}} \prod_{v\in\Delta_0} V_v
\prod_{e\in\Delta_1} (-1)^{s_e} E_e(\pi_e)
\prod_{f\in\Delta_2} (-1)^{s_f} F_f(\lambda_f)
\end{equation}
where only edge and face terms have an undetermined sign.
The dependence on $\xi$ is only in $\mu$
(different $\tn$ giving the same result) as well as the weights.

\subsection{Virtual dimension}

In the case of a CY fourfold, there are new contributions from quadratic
terms.   Each vertex now contributes
\begin{equation}
\vdim_v = |K_{v,reg}| - \sum_{(abcd)=(1234),(1324),(1423)}
|\lambda_{ab}| \cdot |\lambda_{cd}|
\end{equation}
each edge contributes
\begin{equation}
\vdim_e=f_e(\pi)-
\sum_{\substack{a,b\neq e\\ a\neq b}}
\sum_{(i,j)\in\lambda_{ae}}
\sum_{(i',k')\in\lambda_{be}}
\Bigl(i\psi_i+j\psi_j-(k'\psi_k+i'\psi_i)\Bigr),
\end{equation}
where $f_e(\pi)$ is the standard finite edge/framing contribution for a
solid partition $\pi$ in the notation of \cite{Nekrasov:2023nai}.  The
face contribution was computed in \cite[Lemma~9.6]{Nekrasov:2023nai}.

\subsection{Higher rank: the colored/Quot sector}

The discussion in this subsection concerns the colored direct-sum (equivalently, framed/Quot-type local) sector, not the full moduli space of stable higher-rank sheaves.  Introduce a fixed vector space $\tN$ of dimension $r$ and a rank-$r$ equivariant bundle $\tM$, with characters
$\tN=\sum_{i=1}^r\ta_i$ and $\tM=\sum_i\tm_i$.  At a colored fixed point we take
$F=\bigoplus_{i=1}^rF_i$ with $F_i$ rank-one ideal-sheaf data.  In the no-leg zero-dimensional case on $\BC^4$, this is precisely the higher-rank Quot geometry whose K-theoretic virtual weights and sign rule were derived geometrically by Kool--Rennemo \cite{KoolRennemo:2025}.  With nontrivial asymptotics, the global orientation problem for the assembled vertex/edge/face theory remains separate.

The full self-dual character at a torus fixed point reads
\begin{equation}
\chi(F,F) = \sum_{v\in\Delta_0} \frac{\ch (E_v\otimes E_v^*)} {P_{1234}}
\end{equation}
where $\ch E_v = \tN-P_{1234}\tK$, and $\tK$ is a tuple of $r$ asymptotic
solid partitions,
colored by $\tN$, i.e.\ $\tK = \sum_{i=1}^r \ta_i \tK_i$.
Compatibility is imposed independently for each $\tK_i$.

Our observable is obtained by replacing 
$E_v$ with $E_v - \tM$ in the RHS above.
Denote the resulting full self-dual character (after the indicated constant subtraction) by
\begin{equation}
-T^{(2)} = \frac{\ch (E_v-\tM) \ch (E_v-\tM)^*} {P_{1234}}
- \frac{(\tN-\tM)(\tN-\tM)^*}{P_{1234}}
\end{equation}
so that $-\chi_0(F,F)=\sum_{v\in\Delta_0}T^{(2)}|_{\tM=0}$.
The minus sign is due to our conventions for the plethystic exponential.

For $i=1,\dots,r$, let $\tn_i$ be the equivariant first Chern class of $F_i$.
The local divisor asymptotics are non-negative multiplicities, subject to the toric divisor relations; globally they determine equivariant line bundles $L_i$ with character $\ch L_i=q^{\tn_i}$.  We write $L_i=\mathcal O_Y(\sum_{f\in\Delta_{3}}\tn_{i,f}D_f)$ for a chosen toric presentation.
Use it to define the twist
\begin{equation}
K_i = q^{-\tn_i} \left( \tK_i - \frac{1-q^{\tn_i}}{P_{1234}} \right)
\end{equation}
One can check that these new solid partitions can only be infinite along curves and
surfaces \cite{Nekrasov:2023nai}.  However, the corresponding
equivariant parameters $a_i = q^{\tn_i} \ta_i$ can depend on the fixed point $v\in\Delta_0$,
via the integers $\tn_i$.
This construction satisfies $N-P_{1234}K=\tN-P_{1234}\tK$,
where $N = \sum_i a_i$ and $K = \sum_i a_i K_i$.
Thus
\begin{equation}\label{eq:higher-full-character}
-T^{(2)} = \frac{(N-P_{1234}K-\tM)(N-P_{1234}K-\tM)^*-(\tN-\tM)(\tN-\tM)^*}{P_{1234}}.
\end{equation}
Define the residual perturbative self-dual character by
\begin{equation}\label{eq:T6-full}
-T_6^{(2)} = \sum_{v\in\Delta_0} \frac {NN^*-\tN\tN^* - \tM (N - \tN)^* - \tM^* (N - \tN) } {P_{1234}}
\end{equation}
This is zero in the case of zero first Chern class ($\tn_i=0$ for all
$i=1,\dots,r$ and all divisors), and a Laurent polynomial in general,
depending only on the integers $\tn_i$.  Explicitly:
\begin{equation}
-T_6^{(2)} = \sum_v  \frac{
\sum_{i \neq j} \ta_{ij} (q^{\tn_{ij}}-1)
- \sum_{i,j} \tm_i/\ta_j (q^{-\tn_{j}}-1)
- \sum_{i,j} \ta_{i}/\tm_j (q^{\tn_{i}}-1)} {P_{1234}}
\end{equation}
Choose a half-character $T_6$ satisfying
\begin{equation}\label{eq:T6-half}
 T_6+T_6^*=T_6^{(2)}.
\end{equation}
Such a choice is obtained by pairing the second and third sums in \eqref{eq:T6-full} and splitting the first sum into $i<j$ and $i>j$ parts; the remaining sign is part of the DT4 orientation data.  Using
$P_{123}+P_{123}^*=P_{1234}$, define the variable half-character
\begin{equation}\label{eq:mag4}
T_{\mathrm{var}}=(N-\tM)K^*-P_{123}KK^*.
\end{equation}
Then the total half-character is
\begin{equation}\label{eq:total-higher-half}
 T_{\mathrm{half}}=T_6+\sum_{v\in\Delta_0}T_{\mathrm{var},v},
 \qquad
 T_{\mathrm{half}}+T_{\mathrm{half}}^*=\sum_{v\in\Delta_0}T_v^{(2)}.
\end{equation}
Thus the unnormalized localization weight contains the factor $\hat a(T_6)$.  Since $T_6$ depends only on the fixed flux sector, one may divide by this common perturbative factor and work with the normalized variable part $\hat a(\sum_vT_{\mathrm{var},v})$; this normalization must be stated explicitly.

\subsection{Reduction to threefold}

Let $Y=\operatorname{Tot}(K_X)$ and let $q_4$ denote the fiber weight.  In the colored vertex sector we specialize the insertion to $\tM=N\otimes K_X$, so that $\det(N-\tM)=K_X^{-r}$.  In the normalized variable half-character \eqref{eq:mag4}, the $\tM$-independent part is movable, whereas the insertion term contributes a negative character proportional to $q_4^{-1}K$.  If
\[
 K_i=\sum_{n=0}^h q_4^n\,\pi_{3,n}(q_1,q_2,q_3)
\]
has a nonzero fiber layer with $n>0$, the $K_X$-twisted exterior-algebra
insertion acquires a trivial-weight factor.  Its localized Euler factor
therefore vanishes, so only colored fixed points represented by plane
partitions on the zero section survive.  This is the fixed-point
manifestation of the support section considered in the rank-one global
argument below.

This is a local vertex statement.  It does not, by itself, identify the full moduli space of stable rank-$r$ sheaves on $Y$ with a moduli space on $X$.  The global scheme-theoretic argument below concerns rank-one ambient ideals and is kept logically separate.

\subsubsection{Rank-one scheme-theoretic and virtual reduction to the zero section}

This subsection is entirely rank one.  It supplies a global geometric explanation for the support insertion on Hilbert schemes, but it is not a proof of a higher-rank stable-sheaf reduction.

Let $p\colon Y=K_X\to X$ be the bundle projection and let
$i\colon X\hookrightarrow Y$ be the zero section.  The zero section is an
effective Cartier divisor, cut out by the tautological section of
$p^*K_X$.  Hence
\begin{equation}\label{eq:zero-section-line-bundles}
 \mathcal I_{X/Y}=\mathcal O_Y(-X)\cong p^*K_X^{-1},
 \qquad
 \mathcal O_Y(X)\cong p^*K_X,
 \qquad
 N_{X/Y}\cong K_X.
\end{equation}
The first equality in \zcref{eq:zero-section-line-bundles} involves the
inverse of the line bundle used in the insertion above: the divisor $X$ is
cut out by a section of $\mathcal O_Y(X)\cong p^*K_X$, whereas its ideal is
$\mathcal O_Y(-X)\cong p^*K_X^{-1}$.

The following statement is the fixed-divisor analogue of
\cite[Proposition~2.1 and diagram~(7)]{GholampourSheshmaniThomas:2014}.

\begin{proposition}[Fixed-divisor Hilbert functor]\label{prop:fixed-divisor-hilbert}
For a scheme $S$, put
\[
 Y_S=Y\times_{\operatorname{Spec}\mathbb C}S,
 \qquad
 X_S=X\times_{\operatorname{Spec}\mathbb C}S,
 \qquad
 i_S=i\times\operatorname{id}_S.
\]
Let $\Hilb_X(Y)$ be the functor whose $S$-valued points are $S$-flat
families $\mathcal Z\subset Y_S$ which are scheme-theoretically contained
in the fixed divisor $X_S$.  Then the natural transformation
\begin{equation}\label{eq:hilb-functor-isomorphism}
 \Hilb(X)\longrightarrow \Hilb_X(Y),
 \qquad
 \mathcal Z\subset X_S\longmapsto
 \mathcal Z\subset X_S\subset Y_S,
\end{equation}
is an isomorphism of functors.  In particular, it is an isomorphism of the
corresponding Hilbert schemes.

For every such family there is a canonical short exact sequence
\begin{equation}\label{eq:ambient-intrinsic-ideal-sequence}
 0\longrightarrow \mathcal I_{X_S/Y_S}
 \longrightarrow \mathcal I_{\mathcal Z/Y_S}
 \longrightarrow i_{S*}\mathcal I_{\mathcal Z/X_S}
 \longrightarrow 0.
\end{equation}
If
$\rho_S\colon\mathcal O_{Y_S}\twoheadrightarrow i_{S*}\mathcal O_{X_S}$
is the fixed restriction map, then the ambient ideal is recovered from the
intrinsic ideal by the ordinary inverse-image, or equivalently kernel,
formula
\begin{equation}\label{eq:ambient-ideal-kernel}
 \mathcal I_{\mathcal Z/Y_S}
 =\rho_S^{-1}\!\left(i_{S*}\mathcal I_{\mathcal Z/X_S}\right)
 =\ker\!\left(
 \mathcal O_{Y_S}\xrightarrow{\rho_S}i_{S*}\mathcal O_{X_S}
 \longrightarrow i_{S*}\mathcal O_{\mathcal Z}
 \right).
\end{equation}
\end{proposition}

\begin{proof}
A family $\mathcal Z\subset Y_S$ is scheme-theoretically contained in
$X_S$ if and only if
\begin{equation}\label{eq:family-containment-condition}
 \mathcal I_{X_S/Y_S}\subset \mathcal I_{\mathcal Z/Y_S}.
\end{equation}
Equivalently, the quotient homomorphism
$\mathcal O_{Y_S}\twoheadrightarrow\mathcal O_{\mathcal Z}$ factors
uniquely through $\rho_S$.  Thus a quotient of $\mathcal O_{Y_S}$ defining
a family contained in $X_S$ is exactly the same datum as a quotient of
$\mathcal O_{X_S}$ defining a family on $X_S$.

More explicitly, for a family $\mathcal Z\subset X_S$, the two standard
ideal sequences fit into the same commutative diagram used in the
moving-divisor argument:
\begin{equation}\label{eq:fixed-divisor-three-by-three}
\begin{CD}
0 @>>> \mathcal I_{X_S/Y_S} @>>> \mathcal I_{\mathcal Z/Y_S}
  @>>> i_{S*}\mathcal I_{\mathcal Z/X_S} @>>> 0\\
@. @| @VVV @VVV @.\\
0 @>>> \mathcal I_{X_S/Y_S} @>>> \mathcal O_{Y_S}
  @>{\rho_S}>> i_{S*}\mathcal O_{X_S} @>>> 0\\
@. @. @VVV @VVV @.\\
@. @. i_{S*}\mathcal O_{\mathcal Z} @= i_{S*}\mathcal O_{\mathcal Z} @.\\
@. @. @VVV @VVV @.\\
@. @. 0 @. 0 @.
\end{CD}
\end{equation}
The central row and the two right-hand columns are the usual ideal
sequences.  The top row is therefore exact; equivalently, it follows by
taking the kernel of the map
$\mathcal O_{Y_S}\to i_{S*}\mathcal O_{\mathcal Z}$.  This proves
\zcref{eq:ambient-intrinsic-ideal-sequence,eq:ambient-ideal-kernel}.

Conversely, let $\mathcal J\subset\mathcal O_{X_S}$ be an ideal whose
quotient is flat over $S$.  Define an ideal on $Y_S$ by
\begin{equation}\label{eq:ambient-ideal-from-intrinsic-ideal}
 \widetilde{\mathcal J}
 \coloneq
 \ker\!\left(
 \mathcal O_{Y_S}\xrightarrow{\rho_S} i_{S*}\mathcal O_{X_S}
 \longrightarrow i_{S*}(\mathcal O_{X_S}/\mathcal J)
 \right).
\end{equation}
Then
$\mathcal I_{X_S/Y_S}\subset\widetilde{\mathcal J}$ and
\[
 \mathcal O_{Y_S}/\widetilde{\mathcal J}
 \cong i_{S*}(\mathcal O_{X_S}/\mathcal J).
\]
Hence the same $S$-flat family is obtained, now viewed inside $Y_S$.
Starting with an ambient ideal satisfying
\zcref{eq:family-containment-condition}, the inverse construction is simply
\[
 i_{S*}\mathcal I_{\mathcal Z/X_S}
 \cong
 \mathcal I_{\mathcal Z/Y_S}/\mathcal I_{X_S/Y_S}.
\]
The two constructions are inverse and commute with base change, because the
quotient $\mathcal O_{\mathcal Z}$ is $S$-flat.  This proves
\zcref{eq:hilb-functor-isomorphism}.
\end{proof}

In particular, the defining section $s_X$ of the fixed divisor gives the
canonical inclusion
\begin{equation}\label{eq:fixed-section-complex}
 \mathcal O_{Y_S}(-X_S)\xrightarrow{s_X}
 \mathcal I_{\mathcal Z/Y_S},
\end{equation}
whose cokernel is $i_{S*}\mathcal I_{\mathcal Z/X_S}$.  Thus the two-term
complex in \zcref{eq:fixed-section-complex} is quasi-isomorphic to
$i_{S*}\mathcal I_{\mathcal Z/X_S}$.  The map is not arbitrary: it is the
restriction of the fixed inclusion
$\mathcal I_{X_S/Y_S}\hookrightarrow\mathcal O_{Y_S}$.  Allowing an
arbitrary map $\mathcal O_{Y_S}(-X_S)\to\mathcal I_{\mathcal Z/Y_S}$ would
allow its zero divisor to vary and would return one to a moving-divisor
problem.

\paragraph{Comparison with the moving-divisor proof.}
In \cite[Proposition~2.1]{GholampourSheshmaniThomas:2014}, one starts from
a universal chain $\mathcal Z\subset\mathcal H$ in a varying linear system.
The top row of their diagram~(7) is exactly the analogue of
\zcref{eq:ambient-intrinsic-ideal-sequence}.  The difficult part of their
converse is to recover both $\mathcal Z$ and the moving divisor
$\mathcal H$ from a universal stable pair: one takes the double dual of the
rank-one sheaf, accounts for a line bundle pulled back from the base, and
then recovers $\mathcal H$ as the zero locus of the induced section.  None
of those steps is needed here.  The divisor $X_S$ and its restriction map
$\rho_S$ are fixed from the outset, so the inverse construction is simply
the quotient by $\mathcal I_{X_S/Y_S}$, or equivalently the kernel formula
\zcref{eq:ambient-ideal-from-intrinsic-ideal}.

All occurrences of $\times$ in the definitions of $X_S$ and $Y_S$ above
are Cartesian products of schemes over $\operatorname{Spec}\mathbb C$.
The reconstruction formulas
\zcref{eq:ambient-ideal-kernel,eq:ambient-ideal-from-intrinsic-ideal} use
kernels and inverse images of sheaves; they are not tensor products.  Tensor
products are denoted by $\otimes$ throughout.

At the level of ordinary moduli functors, Proposition~\ref{prop:fixed-divisor-hilbert}
therefore gives the desired identification.  As in the moving-divisor paper,
the comparison of virtual counts is a separate statement: one must compare
the obstruction theories, or realize the fixed-support locus as the virtual
zero locus of the corresponding tautological section.  This is the role of
the insertion considered below.

\paragraph{Normal deformations and the role of the insertion.}
The support condition is essential.  The tangent space to the full Hilbert
scheme of $Y$ at $Z$ is
\begin{equation}
 T_Z\Hilb(Y)=\Hom_Y(\mathcal I_{Z/Y},\mathcal O_Z).
\end{equation}
The differential of the condition that $Z$ remain inside $X$ is the
restriction map
\begin{equation}\label{eq:normal-tangent-map}
 \Hom_Y(\mathcal I_{Z/Y},\mathcal O_Z)
 \longrightarrow
 \Hom_Y(\mathcal O_Y(-X),\mathcal O_Z)
 \cong H^0(Z,\mathcal O_Z(X))
 \cong H^0(Z,K_X|_Z).
\end{equation}
Its kernel consists of homomorphisms which vanish on
$\mathcal I_{X/Y}$ and hence factor through
$i_*\mathcal I_{Z/X}$; therefore
\begin{equation}
 T_Z\Hilb_X(Y)
 \cong \Hom_X(\mathcal I_{Z/X},\mathcal O_Z).
\end{equation}
Thus the rigidity of the divisor, measured by $H^0(X,N_{X/Y})$, does not by
itself eliminate normal deformations of a subscheme, which are measured by
$H^0(Z,N_{X/Y}|_Z)$.  For example, for a point $x\in X$ one has
$\Hilb^1(Y)=Y$ whereas $\Hilb^1_X(Y)=X$, and the missing tangent direction
is precisely the fiber of $K_X$ at $x$.

Let $\mathcal Z\subset\cm_\gamma\times Y$ be the universal subscheme and
let $\pi_1,\pi_2$ be the two projections.  The defining section
$s_X\in H^0(Y,\mathcal O_Y(X))$ induces the tautological section
\begin{equation}\label{eq:tautological-support-section}
 \tau_X\colon \mathcal O_{\cm_\gamma}\longrightarrow
 \mathrm R\pi_{1*}
 \bigl(\mathcal O_{\mathcal Z}\otimes\pi_2^*\mathcal O_Y(X)\bigr).
\end{equation}
For an $S$-valued family, the pullback of $\tau_X$ is the restriction of
$s_X$ to that family.  Its vanishing is equivalent to $s_X$ acting by zero on
$\mathcal O_{\mathcal Z}$, or equivalently to
$\mathcal I_{X_S/Y_S}\subset\mathcal I_{\mathcal Z/Y_S}$.  Thus, at the functorial level, the vanishing condition is exactly $\Hilb_X(Y)\cong\Hilb(X)$.  When the derived pushforward in \zcref{eq:tautological-support-section} is represented by a vector bundle (or by a perfect complex of the amplitude required by the DT4 Lefschetz formalism), this statement is realized by the corresponding ordinary/derived zero locus.

Under the hypotheses of the DT4 K-theoretic Lefschetz principle---perfectness and the required amplitude of the tautological complex, properness of the relevant fixed loci in the noncompact equivariant theory, compatibility of the symmetric obstruction theories, and compatible orientations---the localized virtual-pullback identity takes the schematic form
\begin{equation}\label{eq:k-theoretic-lefschetz-local}
 \sum_{e}(-1)^{\sigma(e)}(j_e)_*
 \widehat{\mathcal O}^{\mathrm{vir}}_{\Hilb(X)_e}
 =
 \widehat e\!\left(
 \mathrm R\pi_{1*}
 \bigl(\mathcal O_{\mathcal Z}\otimes\pi_2^*\mathcal O_Y(X)\bigr)
 \right)
 \otimes \widehat{\mathcal O}^{\mathrm{vir}}_{\cm_\gamma}.
\end{equation}
Here $j_e$ denotes the inclusion of a connected component, $\widehat e$ is the symmetrized K-theoretic Euler class, and the signs encode orientations.  Park proved the corresponding DT4 virtual-pullback/Lefschetz principle at the level of Oh--Thomas virtual cycles under a compatibility condition on the symmetric obstruction theories \cite{Park:2021}.  K-theoretic DT4 Lefschetz principles, including toric reductions to three-dimensional DT/PT theories for suitable tautological insertions, are developed by Bae--Kool--Park \cite{BaeKoolPark:2024}.  Formula \eqref{eq:k-theoretic-lefschetz-local} is to be read under precisely these compatibility and amplitude hypotheses; it is not a formal consequence of Proposition~\ref{prop:fixed-divisor-hilbert} alone.

Since $\mathcal O_Y(X)\cong p^*K_X$, this gives the global rank-one reason for using a $K_X$-twisted support insertion.  Its toric localization is compatible with the local observation that fixed subschemes carrying a nonzero fiber layer are killed by the insertion.

The universal ideal sequence
$0\to\mathcal I_{\mathcal Z/Y}\to\mathcal O_Y\to
\mathcal O_{\mathcal Z}\to0$ also gives, in equivariant $K$-theory,
\begin{equation}\label{eq:quotient-versus-ideal-insertion}
 \left[\mathrm R\pi_{1*}(\mathcal O_{\mathcal Z}(X))\right]
 =
 \left[\mathrm R\Gamma(Y,\mathcal O_Y(X))\right]
 -
 \left[\mathrm R\pi_{1*}(\mathcal I_{\mathcal Z/Y}(X))\right].
\end{equation}
Thus the divisor-cutting Euler class may equivalently be written, up to the
fixed perturbative factor coming from
$\mathrm R\Gamma(Y,\mathcal O_Y(X))$, in terms of the universal ideal
sheaf.  This is the normalization naturally used in the vertex expression.

\paragraph{Relation with the pushed-forward ideal sheaf.}
One should finally distinguish $\mathcal I_{Z/Y}$ from
$i_*\mathcal I_{Z/X}$: the former has rank one on $Y$, whereas the latter
has rank zero and is supported on $X$.  They lie in different sheaf-moduli
problems, and \zcref{eq:ambient-intrinsic-ideal-sequence} gives a relation
between them rather than an isomorphism.  Nevertheless, there is a separate
DT4/DT3 dimensional-reduction statement for the pushed-forward sheaf.  If
$F=\mathcal I_{Z/X}$, the Koszul resolution of the zero section gives
\begin{equation}\label{eq:pushforward-rhom-splitting}
 \begin{aligned}
 \mathrm R\Hom_Y(i_*F,i_*F)
 &\simeq
 \mathrm R\Hom_X(F,F)
 \oplus
 \mathrm R\Hom_X(F,F\otimes K_X)[-1]\\
 &\simeq
 \mathrm R\Hom_X(F,F)
 \oplus
 \mathrm R\Hom_X(F,F)^\vee[-4].
 \end{aligned}
\end{equation}
If $\Hom_X(F,F\otimes K_X)=0$, as happens for the stable sheaves in the
standard Fano situation, then
\begin{align}
 \Ext^1_Y(i_*F,i_*F)&\cong \Ext^1_X(F,F),\\
 \Ext^2_Y(i_*F,i_*F)&\cong
 \Ext^2_X(F,F)\oplus\Ext^2_X(F,F)^\vee.
\end{align}
The first summand in the second line is maximal isotropic for the CY4 Serre pairing.  The identification of the Kuranishi map with the threefold Kuranishi map is an additional compatibility statement; in the standard local-canonical-bundle setting it is precisely the content of the DT4/DT3 dimensional-reduction result of \cite{CaoLeung:2014}.  After choosing the compatible complex orientation, the DT4 half-obstruction theory then reduces to the ordinary DT3 obstruction theory on $X$.  This second statement justifies counting
$i_*\mathcal I_{Z/X}$ by the threefold theory, while
\zcref{eq:hilb-functor-isomorphism,eq:k-theoretic-lefschetz-local} justify
counting the rank-one ambient ideals $\mathcal I_{Z/Y}$ with the support
insertion.  The two reductions are compatible, but they are logically
distinct.

\section*{Acknowledgments}
The authors thank Nikita Nekrasov and Martijn Kool for fruitful conversations.  The second author acknowledges support from the Beijing Institute of Mathematical Sciences and Applications (BIMSA), the Beijing NSF BJNSF-IS24005, the NSFC-RFIS program W2432008, the NSF AI Institute for Artificial Intelligence and Fundamental Interactions at MIT under Cooperative Agreement PHY-2019786, and China's National Program of Overseas High Level Talent.  The research of NP is partially supported by the Beijing Natural Science Foundation under grant IS26010.


\begin{thebibliography}{99}
\bibitem{BaeKoolPark:2024} Y.~Bae, M.~Kool, and H.~Park, \emph{Counting surfaces on Calabi--Yau 4-folds II: DT--PT$_0$ correspondence}, arXiv:2402.06526 (2024).
\bibitem{CaoLeung:2014} Y.~Cao and N.~C.~Leung, \emph{Donaldson--Thomas theory for Calabi--Yau 4-folds}, arXiv:1407.7659 (2014).
\bibitem{Corrigan:1982th} E.~Corrigan, C.~Devchand, D.~Fairlie, and J.~Nuyts, \emph{First order equations for gauge fields in spaces of dimension greater than four}, Nucl. Phys. B 214 (1983), 452--464.
\bibitem{DelZotto:2021gzy} M.~Del Zotto, N.~Nekrasov, N.~Piazzalunga, and M.~Zabzine, \emph{Playing with the index of M-theory}, Commun. Math. Phys. 396 (2022), 817--865; arXiv:2103.10271.
\bibitem{Donaldson:1996kp} S.~K.~Donaldson and R.~P.~Thomas, \emph{Gauge theory in higher dimensions}, in \emph{The Geometric Universe}, Oxford Univ. Press (1998), 31--47.
\bibitem{Duistermaat:1982vw} J.~J.~Duistermaat and G.~J.~Heckman, \emph{On the variation in the cohomology of the symplectic form of the reduced phase space}, Invent. Math. 69 (1982), 259--268.
\bibitem{FasolaMonavariRicolfi:2021} N.~Fasola, S.~Monavari, and A.~T.~Ricolfi, \emph{Higher rank K-theoretic Donaldson--Thomas theory of points}, Forum Math. Sigma 9 (2021), e15; arXiv:2003.13565.
\bibitem{FeyzbakhshThomasRank0:2024} S.~Feyzbakhsh and R.~P.~Thomas, \emph{Rank $r$ DT theory from rank 0}, Duke Math. J. 173 (2024), 2063--2116; arXiv:2103.02915.
\bibitem{Feyzbakhsh:2021nds} S.~Feyzbakhsh and R.~P.~Thomas, \emph{Rank $r$ DT theory from rank 1}, J. Amer. Math. Soc. 36 (2023), 795--826; arXiv:2108.02828.
\bibitem{GholampourKoolYoung:2018} A.~Gholampour, M.~Kool, and B.~Young, \emph{Rank 2 sheaves on toric 3-folds: classical and virtual counts}, Int. Math. Res. Not. IMRN 2018, no.~10, 2981--3069.
\bibitem{GholampourSheshmaniThomas:2014} A.~Gholampour, A.~Sheshmani, and R.~P.~Thomas, \emph{Counting curves on surfaces in Calabi--Yau 3-folds}, Math. Ann. 360 (2014), 67--78; arXiv:1309.0051.
\bibitem{KarpovMoreira:2026} I.~Karpov and M.~Moreira, \emph{On the Feyzbakhsh--Thomas programme for Fano 3-folds}, arXiv:2607.13982 (2026).
\bibitem{karshon-lerman:2015} Y.~Karshon and E.~Lerman, \emph{Non-compact symplectic toric manifolds}, SIGMA 11 (2015), 055.
\bibitem{KoolRennemo:2025} M.~Kool and J.~V.~Rennemo, \emph{Proof of a magnificent conjecture}, arXiv:2507.02852 (2025).
\bibitem{Maulik:2003rzb} D.~Maulik, N.~Nekrasov, A.~Okounkov, and R.~Pandharipande, \emph{Gromov--Witten theory and Donaldson--Thomas theory, I}, Compos. Math. 142 (2006), 1263--1285; arXiv:math/0312059.
\bibitem{Nekrasov:2014nea} N.~Nekrasov and A.~Okounkov, \emph{Membranes and sheaves}, Algebr. Geom. 3 (2016), 320--369; arXiv:1404.2323.
\bibitem{Nekrasov:2018xsb} N.~Nekrasov and N.~Piazzalunga, \emph{Magnificent Four with colors}, Commun. Math. Phys. 372 (2019), 573--597; arXiv:1808.05206.
\bibitem{Nekrasov:2023nai} N.~Nekrasov and N.~Piazzalunga, \emph{Global Magni4icence, or: 4G Networks}, SIGMA 20 (2024), 106; arXiv:2306.12995.
\bibitem{Oh:2020rnj} J.~Oh and R.~P.~Thomas, \emph{Counting sheaves on Calabi--Yau 4-folds, I}, Duke Math. J. 172 (2023), 1333--1409; arXiv:2009.05542.
\bibitem{Park:2021} H.~Park, \emph{Virtual pullbacks in Donaldson--Thomas theory of Calabi--Yau 4-folds}, arXiv:2110.03631 (2021).
\end{thebibliography}
\end{document}